\documentclass[12pt]{amsart}

\usepackage{amsthm,amsfonts,amsbsy,amssymb,amsmath,amscd}
\usepackage[all]{xy}
\usepackage[colorlinks=true]{hyperref}

\newcommand{\Vol}{\rm Vol}

\newcommand{\Alb}{\rm Alb}

\begin{document}
\title{Geography of irregular Gorenstein 3-folds}
\author{Tong Zhang}
\date{\today}

\address{Department of Mathematics, University of Alberta, Edmonton, Alberta T6G 2G1, Canada}
\email{tzhang5@ualberta.ca}


\begin{abstract}
In this paper, we study the explicit geography problem of irregular Gorenstein minimal 3-folds of general type. We generalize the classical Noether-Castelnuovo
inequalities for irregular surfaces to irregular 3-folds according to the Albanese dimension.
\end{abstract}

\maketitle

\theoremstyle{plain}
\newtheorem{theorem}{Theorem}[section]
\newtheorem{lemma}[theorem]{Lemma}
\newtheorem{coro}[theorem]{Corollary}
\newtheorem{prop}[theorem]{Proposition}
\newtheorem{defi}[theorem]{Definition}
\newtheorem{ques}[theorem]{Question}
\newtheorem{conj}[theorem]{Conjecture}

\theoremstyle{remark}
\newtheorem{remark}[theorem]{\bf Remark}
\newtheorem{assumption}[theorem]{\bf Assumption}
\newtheorem{example}[theorem]{\bf Example}
\newtheorem{obser}[theorem]{\bf Observation}

\tableofcontents

\section{Introduction}
We work over an algebraically closed field of characteristic $0$.

A projective variety $X$ is called \textit{irregular}, if $h^1(\mathcal O_X)>0$, i.e.,
$X$ has a nontrivial Albanese map. Denote by $a(X) \subseteq {\Alb}(X)$
the image of $X$ under its Albanese map. The Albanese dimension $\dim a(X)$ can vary from one to $\dim X$. We say
$X$ is \emph{of Albanese dimension $m$}, if $\dim a(X)=m$. In particular, we say
$X$ is \emph{of maximal Albanese dimension}, if $\dim a(X)=\dim X$.

The purpose of this paper is to study the geography
problem of irregular varieties.

Let $C$ be a projective curve of genus $g > 0$. One has
$$
\deg(\omega_C)=2\chi(\omega_C) \ge 0.
$$

The above result has several 2-dimensional generalizations.
For an irregular minimal surface $S$ of general type (with ADE singularities), $\chi(\omega_S)>0$. One has the Noether, Castelnuovo and Severi inequalities for \textit{irregular} surfaces proved
respectively by Bombieri \cite{Bo}, Horikawa \cite{Ho2} and Pardini \cite{Pa}.
Namely,
\begin{itemize}
\item[(1)] Noether inequality: $K^2_S \ge 2 \chi(\omega_S)$ if $S$ is irregular (see \cite{Bo});
\item[(2)] Castelnuovo inequality: $K^2_S \ge 3 \chi(\omega_S)$ if the Albanese fiber is not hyperelliptic of genus $2$ or $3$ (see \cite{Ho2});
\item[(3)] Severi inequality: $K^2_S \ge 4 \chi(\omega_S)$ if $S$ is of maximal Albanese dimension (See \cite{Se, Pa}).
\end{itemize}
The above results concern the geography of irregular surfaces (of Albanese fiber dimension $\le 1$) and they have played a very important role in the surface theory. Also, in the recent work of Lu \cite{Lu1, Lu2} and Lopes-Pardini \cite{MP2}, results of such type have been applied to study the hyperbolicity of irregular surfaces.

Having seen the importance of such results, as a natural question, one can ask
\begin{ques} \label{question1}
What are the Noether, Castelnuovo and Severi inequalities for  irregular Gorenstein 3-folds?
\end{ques}

Here we assume the 3-folds being Gorenstein so that $\chi(\omega_X)>0$ and we can have nontrivial inequalities
$$
K^3_X \ge a \chi(\omega_X)
$$
with $a>0$ similar to the surface case.

In fact, several questions of the similar type have been raised before.
In early 1980's, Miles Reid asked the following question: \textit{what is the Noether inequality for 3-folds}?
Also, as an open problem in \cite[\S 3.9]{Ch}, Chen conjectured that
\begin{conj} \label{question2}
For Gorenstein minimal 3-fold $X$ of general type,
there should be a Noether inequality in the following form:
$$
K^3_X \ge a \chi(\omega_X) - b
$$
where $a (>1)$, $b$ are both positive rational numbers.
\end{conj}

As is mentioned in \cite{Ch}, \textit{any bound $a>1$ is nontrivial and interesting}. However, it might be
more difficult than the inequality between $K^3_X$ and $p_g(X)$. One possible problem comes from
the difficulty to understand $h^1(\mathcal O_X)$ and $h^2(\mathcal O_X)$. Another problem may be due to the non-smoothness.
For example, when $X$ is smooth and minimal, it is proved in \cite{CaMChZ} that $K^3_X \ge \frac 23(2p_g(X)-5)$. From this, Chen
and Hacon \cite{MChH} have proved that $a=\frac 89$. But if $X$ is Gorenstein, it is still an open question whether $K^3_X \ge \frac 23(2p_g(X)-5)$
holds. See \cite[Conjecture 4.4]{CaMChZ}.

Recently, the Severi inequality has been proved by Barja \cite{Ba} and by the author \cite{Zh} independently: Let $X$ be an irregular minimal Gorenstein 3-fold of general type. If $X$ is of maximal Albanese dimension, then
$$
K^3_X \ge 12 \chi(\omega_X).
$$

Our first purpose of this paper is to give a complete answer to Question \ref{question1} by proving the following Noether-Castelnuovo inequalities.

\begin{theorem} \label{3fold}
Let $X$ be an irregular minimal Gorenstein 3-fold of general type.
If $X$ has Albanese fiber dimension one,
then
$$
K^3_X \ge 4 \chi(\omega_X).
$$
Moreover, if the Albanese fiber is not hyperelliptic of genus $\le 5$, then
$$
K^3_X \ge 6 \chi(\omega_X).
$$
\end{theorem}

This theorem, combining with the Miyaoka-Yau inequality $K^3_X \le 72 \chi(\omega_X)$, will give the explicit geography of Gorenstein irregular 3-folds of Albanese dimension $\ge 2$.

Along this line, one can naturally consider the following much finer conjecture of Noether-Castelnuovo type when
$X$ has Albanese dimension one:
\begin{conj} \label{nc}
Let $X$ be an irregular minimal Gorenstein 3-fold of general type with Albanese dimension one. Then
$$
K^3_X \ge 2 \chi(\omega_X).
$$
If the Albanese fiber has large volume,
then
$$
K^3_X \ge 3 \chi(\omega_X).
$$
\end{conj}
Here we consider the volume of the fiber instead of its geometric genus simply because
the volume is always positive for a surface of general type, but the geometric genus is not.

One should note that the above conjecture is the most suitable generalization by comparing the coefficients before $\chi(\omega_X)$ to those in the surface case or in Theorem \ref{3fold}.

By the sharp inequality $K^3_X \ge \frac 23(2p_g(X)-5)$ \cite{CaMChZ} in the smooth case, Conjecture \ref{nc} seems to be too optimistic. But surprisingly, the following theorem shows that the above conjecture is at most \textit{only a little} away from being true.

\begin{theorem} \label{alb1}
Let $X$ be an irregular minimal Gorenstein 3-fold of general type with Albanese dimension one. Let
$f: X \to Y$ be the induced Albanese fibration with a smooth general fiber $F$. Then
$$
K^3_X \ge 2 \chi(\omega_X)
$$
unless one of the following holds:
\begin{itemize}
\item $p_g(F)=2$ and $K^2_F = 1$. In this case, $K^3_X \ge \frac 43 \chi(\omega_X)$.
\item $p_g(F)=3$ and $K^2_F = 2$. In this case, $K^3_X \ge \frac {12}{7} \chi(\omega_X)$.
\end{itemize}

If $K^2_F \ge 12$, then
$$
K^3_X \ge 3 \chi(\omega_X).
$$
\end{theorem}

Recall that in the surface case, we have $K^2_S \ge 3\chi(\omega_S)$ if the Albanese fiber is nonhyperelliptic. Here, we also study the irregular 3-folds whose Albanese fiber has no hyperelliptic pencil. We have the following result.

\begin{theorem} \label{alb12}
Notations as in Theorem \ref{alb1}. Suppose $F$ has no hyperelliptic pencil. Then
$$
K^3_X \ge 3 \chi(\omega_X).
$$
provided that $K^2_F \ge 9$.
\end{theorem}

\begin{remark}
The geography of non-Gorenstein 3-folds of general type is a very subtle topic. In particular,
$\chi(\omega_X)$ can be zero or even negative. For example, there do exist examples of non-Gorenstein
3-folds of maximal Albanese dimension with $\chi(\omega_X)=0$ (see \cite{EL}). In \cite{CH}, Chen and Hacon have constructed
a family of non-Gorenstein 3-folds of general type with $\chi(\omega_X)$ negative. In their paper, they obtained a similar
type of inequality
$$
K^3_X \ge c \chi(\omega_X)
$$
but $c<0$.
One can also construct families of non-Gorenstein 3-folds of general type with Albanese dimension two and $\chi(\omega_X)<0$ (see Example \ref{example}).
In these cases, Theorem \ref{3fold}
holds trivially. We would like to point out that in \cite{CC}, Chen and Chen have proved that there exists an explicit effective lower bound for ${\Vol} (X)$.
\end{remark}

Let us sketch the proofs of the above theorems. If $X$ has Albanese dimension two, Pardini's method \cite{Pa} on \'etale covering and limiting can be applied here, provided one has a good slope inequality for fibered 3-folds over surfaces, which is not known yet. In this paper, to overcome this difficulty, we prove the relative Noether inequalities such as Theorem \ref{relnoether1}, \ref{relnoether2}, and \ref{essential} for fibered 3-folds over surfaces. These inequalities are about $K^3_X$ and $h^0(K_X)$ up to some explicit error terms. Then by the generic vanishing theorem due to Green and Lazarsfeld \cite{GL}, we know that $\chi(\omega_X)$ is bounded from above by $h^0(K_X)$ up to \'etale covering. Finally, by Pardini's limiting trick, we can prove Theorem \ref{3fold}.

If the Albanese dimension of $X$ is one, we still have the \'etale covering method by Bombieri and Horikawa \cite{Bo, Ho}. But the generic vanishing does not help in this case. Alternatively, we will prove the following relative Noether inequalities, Theorem \ref{relativenoether3fold1} and \ref{relativenoether3fold2}, for fibered 3-folds over curves. They will imply Theorem \ref{alb1} and \ref{alb12} via the above covering method if the volume of the Albanese fiber $\ge 4$. Note that the slope inequality for fibered 3-folds over curves has been studied by Ohno \cite{Oh} and Barja \cite{Ba1}. Finally, we will prove case by case when the volume of the Albanese fiber $\le 3$.

This paper is organized as follows: In Section 3, we prove several basic results for fibered 3-folds. In Section 4, we list several results about linear systems on algebraic surfaces. In Section 5, we prove the relative Noether inequalities for fibered 3-folds over surfaces. In Section 6, we will prove Theorem \ref{3fold}. In Section 7 and 8, we consider the case of Albanese dimension one and prove Theorem \ref{alb1} and \ref{alb12}.

After finishing the paper, the author was told by Jungkai Chen that in a very recent paper \cite{CC1} joint with Meng Chen, they proved that $K^3_X \ge \frac 23 (2h^0(K_X)-5)$ still holds in the Gorenstein case. In another very recent paper, Hu \cite{Hu} showed that $K^3_X \ge \frac 43 \chi(\omega_X)-2$ in this case. Also, the author has been informed by Miguel Barja that in \cite[Remark 4.6]{Ba}, the first inequality of Theorem \ref{3fold} is also independently proved by him using a different method.

\smallskip \smallskip
\textbf{Acknowledgement}. The author is very grateful to Miguel Barja, Jungkai Chen, Meng Chen, Xi Chen, Zhi Jiang and Xinyi Yuan for their interests and valuable discussions related to this paper. He has benefited a lot from several discussions with Steven Lu about and beyond this subject since December 2012.
The author thanks Chenyang Xu for his hospitality when the author visited BICMR. Special thanks also goes to ECNU, NUS, USTC for their hospitality.

The work is supported by an NSERC Discovery Grant.

\section{Notations}
The following notations will be frequently used in this paper:

Let $X$ be a projective variety and $L$ be a line bundle on $X$ such that $h^0(L) \ge 2$.
We denote by
$
\phi_{L}: X \dashrightarrow \mathbb P^{h^0(L)-1}
$
the rational map induced by the complete linear series $|L|$. We say that $\phi_L$ is \textit{generically finite},
if $\dim \phi_L(X)=\dim X$. Otherwise, we say $|L|$ is \textit{composed with a pencil}.

A $\mathbb Q$-Weil divisor $D$ on a variety $X$ of dimension $n$ is called \textit{pseudo-effective}, if for any nef line bundles $A_1, \cdots, A_{n-1}$ on $X$, we have
$$
A_1 \cdots A_{n-1} D \ge 0.
$$
Such divisors can be characterized as the limit of effective divisors.

Let $\alpha: X \to A$ be a morphism from $X$ to an abelian variety $A$. Denote by $\mu_d: A \to A$ the multiplicative map of $A$ by $d$.
We have the following diagram:
$$
\xymatrix{X_d \ar[r]^{\phi_d} \ar[d]_{\alpha_d} & X \ar[d]^{\alpha} \\
A \ar[r]^{\mu_d} & A}
$$
Here $X_d=X \times_{\mu_d} A$ is the fiber product. We call $X_d$ \textit{the $d$-th lifting of $X$ by $\alpha$}. In particular, if $\alpha$ is the Albanese map
of $X$, then we call $X_d$ \textit{the $d$-th Albanese lifting of $X$}.
This construction has been used by Pardini in \cite{Pa}.

In this paper, a \textit{fibration} $f: X \to Y$ always means a surjective morphism with connected fibers.

\section{Preliminaries}
Let $X$ be a projective $n$-fold. Here we assume that $n=2, 3$.
Let $f: X \to Y$ be a fibration from $X$ to a smooth projective curve $Y$ with a smooth general fiber $F$.
Then for any nef line bundle $L$ on $X$, we can find a unique integer $e_L$ such that
\begin{itemize}
\item $L-e_LF$ is not nef;
\item $L-eF$ is nef for any integer $e<e_L$.
\end{itemize}
We call this number the \textit{minimum of $L$ with respect to $F$}. In particular, $e_L>0$. Another important fact is that, if $h^0(L-e_LF)>0$, then $|L-e_LF|$ has horizontal base locus (c.f. \cite[\S 2]{Zh}). We have the following theorem.

\begin{theorem} \label{decomposition}
With the above notations. Let $L$ be a nef and effective line bundle on $X$.
Then we have the following quadruples
$$
\{(X_i, L_i, Z_i, a_i), \quad i=0, 1, \cdots, N\}
$$
with the following properties:
\begin{enumerate}
\item $(X_0, L_0, Z_0, a_0)=(X, L, 0, e_L)$.
\item For any $i=0, \cdots, N-1$, $\pi_i: X_{i+1} \to X_i$ is a composition of blow-ups of $X_i$ such that the proper transform of the movable part of $|L_i-a_iF_i|$ is base point free. Here $F_0=F$, $F_{i+1}=\pi^*_{i} F_{i}$ and $a_i=e_{L_i}$ is the minimum of $L_i$ with respect to $F_i$. Moreover, we have the decomposition
$$
|\pi^*_i(L_i-a_iF_i)|= |L_{i+1}| + Z_{i+1}
$$
such that $|L_{i+1}|$ is base point free and $Z_{i+1}|_{F_{i+1}} > 0$.
\item We have $h^0(L_0) > h^0(L_1) > \cdots > h^0(L_N)>h^0(L_N-a_NF_N)=0$. Here $a_N=e_{L_N}$.
\end{enumerate}
\end{theorem}

\begin{proof}
See \cite[\S 2]{Zh}.
\end{proof}

\begin{remark} \label{r0r1}
From the above construction, we have
$$
h^0(L_0|_{F_0}) \ge h^0(L_1|_{F_1}) > h^0(L_2|_{F_2}) > \cdots > h^0(L_N|_{F_N}).
$$
In general, we do \textit{not} know if $h^0(L_0|_{F_0}) > h^0(L_1|_{F_1})$. But if $|L_0|_F|$ is base point free, then
$$
h^0(L_0|_{F_0}) > h^0(L_1|_{F_1}).
$$
This fact will be used in the proofs of Theorem \ref{relativenoether3fold1} and \ref{relativenoether3fold2}.
\end{remark}

Write $\rho_{i}=\pi_{0} \circ \cdots \circ \pi_{i-1}: X_i \to X_0$ for $i=1, \cdots, N$.
Fix a nef line bundle $P=P_0$ on $X$ and
denote
$$
L'_i=L_i-a_iF_i, \quad r_i=h^0(L_i|_{F_i}), \quad d_i=(P_i|_{F_i})(L_i|_{F_i})^{n-1},
$$
where $P_i=\rho^*_i P$. It is easy to see that we have
$$
d_0 \ge d_1 \ge \cdots \ge d_N \ge 0.
$$

\begin{prop} \label{algcase1}
For any $j=0, \cdots, N$, we have the following numerical inequalities:
\begin{itemize}
\item[(1)] $\displaystyle{h^0(L_0) \le h^0(L'_j) + \sum_{i=0}^j a_{i} r_{i}}$;
\item[(2)] $\displaystyle{P_0 L^{2}_0 \ge 2 a_0d_0 + \sum_{i=0}^j a_{i}(d_{i-1}+d_i) - 2d_0}$. \quad $(n=3)$
\end{itemize}
\end{prop}

\begin{proof}
In \cite[\S 2]{Zh}, (1) is proved. We sketch it here. By the following exact sequence
$$
0 \longrightarrow H^0(L_i-F_i) \longrightarrow H^0(L_i) \longrightarrow H^0(L_i|_{F_i}),
$$
we get
$$
h^0(L_i-F_i) \ge h^0(L_i) - h^0(L_i|_{F_i})=h^0(L_i) - r_i.
$$
By induction and summing over $i=0, \cdots, j$, we can get (1).

For (2), since $\pi^*_iL'_i = a_{i+1}F_{i+1} + L'_{i+1} + Z_{i+1}$, we have the following computation of 1-cycles:
\begin{eqnarray*}
(\pi^*_iL'_i)^{2}-L'^{2}_{i+1} & = & (\pi_i^* L'_i -L'_{i+1}) (\pi_i^* L'_i + L'_{i+1}) \\
& = & a_{i+1}(\pi^*_iL'_i+ L'_{i+1}) F_{i+1} + (\pi^*_iL'_i+ L'_{i+1})Z_{i+1} \\
&=& a_{i+1}(\pi^*_iL'_i+ L'_{i+1}) F_{i+1} + (\pi^*_iL'_i + F_{i+1})Z_{i+1} \\
& & + (L'_{i+1}+F_{i+1}) Z_{i+1}- 2 (\pi^*_i L'_i-L'_{i+1}) F_{i+1}.
\end{eqnarray*}
Note that $\pi^*_iL'_i + F_{i+1}$ and $L'_{i+1}+F_{i+1}$ are both nef. Taking intersections with $P_{i+1}$ for both sides, we can get
$$
P_iL^2_i - P_{i+1}L^2_{i+1} = P_{i+1}(\pi^*_iL_i)^2 - P_{i+1}L^2_{i+1} \ge a_{i+1}(d_i+d_{i+1}) - 2(d_i-d_{i-1}).
$$
Summing over $i=0, \cdots, j-1$, we have
$$
P_0 L'^{2}_0-P_jL'^{2}_j \ge \sum_{i=1}^j a_{i}(d_{i-1}+d_i) - 2(d_0-d_j).
$$
Note that we also have
\begin{eqnarray*}
P_0L^{2}_0-P_0 L'^{2}_0 & = & 2 a_0d_0, \\
P_jL'^{2}_j + 2d_j & \ge & 0.
\end{eqnarray*}
Hence (2) follows.
\end{proof}

We can refine the first inequality of Proposition \ref{algcase1} as follows.

\begin{prop} \label{newr}
Let $H^0(L) \to H^0(L|_F)$ be the restriction map. Denote by $r$ the dimension of its image. Then for any $j=1, \cdots, N$, we have
$$
h^0(L_0) \le h^0(L'_j) + a_0 r + \sum_{i=1}^j a_{i} r_{i}.
$$
\end{prop}

\begin{proof}
We know that $r=h^0(L)-h^0(L-F)$. Following the same strategy of the proof of Proposition, one can see that in order to prove the inequality here, we only need to show that
$$
h^0(L-iF) - h^0(L-(i+1)F) \le h^0(L)-h^0(L-F)
$$
for any $0 \le i \le a_0-1$.

In fact, the result holds if $a_0=1$. If $a_0 > 1$, by the following exact sequence
$$
0 \longrightarrow H^0(L-2F) \longrightarrow H^0(L-F) \oplus H^0(L-F) \longrightarrow H^0(L),
$$
we know that
$$
h^0(L-F) - h^0(L-2F) \le h^0(L) - h^0(L-F) = r.
$$
Therefore, we can finish the proof by induction.
\end{proof}

We also have the following lemma.
\begin{lemma}\label{algsumai}
Under the above setting, for $n=2, 3$, we have
$$
P_0L^{n-1}_0 \ge (n-1) d_0(a_0-1)+d_0\sum_{i=1}^N a_i.
$$
\end{lemma}

\begin{proof}
For $i=0, \cdots, N-1$, denote by
$
\tau_i=\pi_{i} \circ \cdots \circ \pi_{N-1} : X_N \to X_i
$
the composition of blow-ups.

Write $b=a_1+\cdots+a_N$ and $Z=\tau^*_1Z_1+ \cdots + \tau^*_{N-1}Z_{N-1} + Z_N$. We have the following numerical equivalence on $X_N$:
$$
\tau^*_0 L'_0 \sim_{\rm{num}} L'_N + bF_{N} + Z.
$$
Since $L'_0+F_0$ and $L'_N+F_N$ are both nef, it follows that
\begin{eqnarray*}
P_0(L'_0+F_0)^{n-1} & = & P_N (\tau^*_0 L'_0 + F_{N})^{n-2}(L'_N+F_N+bF_N+Z) \\
& \ge & b P_N (\tau^*_0 L'_0 + F_{N})^{n-2} F_N \\
& \ge & b d_0.
\end{eqnarray*}
Combining with
$$
P_0L^{n-1}_0-P_0(L'_0+F_0)^{n-1}=(n-1)(a_0-1)d_0,
$$
the proof is finished.
\end{proof}

\section{Linear series on algebraic surfaces}
In this section, we recall some basic results about linear series on algebraic surfaces.
These results will be used to compare the number $r_i$ and $d_i$. They will also serve as
the first step of the induction process.

In this section, we always use the following assumptions:
\begin{itemize}
\item[(1)] $S$ is a smooth algebraic surface of general type with the smooth minimal model $\sigma: S \to S'$;
\item[(2)] $L \ge M$ are two nef line bundles on $S$ such that $L \le K_S$.
\end{itemize}

We have the following Noether type results.
\begin{theorem}  \label{noether}
Assume that $\phi_L$ is generically finite. Then
$$
LM \ge 2h^0(M) - 4.
$$
\end{theorem}
\begin{proof}
To prove this result, we can assume that $h^0(M) \ge 2$.
If $|M|$ is not composed with a pencil, from a result in \cite[Theorem 2]{Sh}, we know that
$$
LM \ge M^2 \ge 2h^0(M) - 4.
$$

If $|M|$ is composed with a pencil, we can write
$$
M \sim_{\rm{num}} r C + Z.
$$
Here $C$ is a general member of the pencil, $r \ge h^0(M) - 1$ and $Z$ is the fixed part of $|M|$.
Because $\phi_L$ is generically finite, $h^0(L|_C) \ge 2$. It implies that $LC \ge 2$ since $S$ is of general type. We get
$$
LM \ge r LC \ge 2h^0(M) - 2.
$$
Hence we prove the theorem.
\end{proof}
The above theorem is also used in \cite{Su}.
\begin{theorem} \label{noether2} \cite[Lemma 2.3]{Oh}
If $|L|$ is not composed with a pencil and $h^0(L) < p_g(S)$, then
$$
(\sigma^*K_{S'})L \ge 2h^0(L)-2.
$$
If $|L|$ is composed with a pencil,
then
$$
(\sigma^*K_{S'})L \ge 2h^0(L) - 2,
$$
except for the case when $K^2_{S'}=1$, $p_g(S)=2$ and $q(S)=0$.
\end{theorem}
We refer to \cite{Ba1} for more of the above type.
We have a better Castelnuovo type bound under some extra conditions.
\begin{theorem} \label{castelnuovo}
Assume that $S$ has no hyperelliptic pencil and that $\phi_L$ is generically finite. Then
$$
LM \ge 3h^0(M) - 7.
$$
\end{theorem}
\begin{proof}
We can assume that $h^0(M) \ge 3$. If $\phi_M$ is generically finite, then by the Castelnuovo type inequality (c.f. \cite[Th\'eor\`eme 5.5]{Be}),
$$
LM \ge M^2 \ge 3h^0(M) - 7.
$$
If not, we can still write $M \sim_{\rm num} r C + Z$
as before. Note that $C$ is not hyperelliptic. Hence $LC \ge 3$. It gives
$$
LM \ge r LC \ge 3h^0(M)-3.
$$
It ends the proof.
\end{proof}

\begin{theorem} \label{castelnuovo2}
Assume that $S$ has no hyperelliptic pencil.
If $|L|$ is composed with a pencil,
then
$$
(\sigma^*K_{S'})L \ge 3h^0(L) - 3,
$$
except for the case when $K^2_{S'}=2$, $p_g(S)=2$ and $q(S)=0$.
If $|L|$ is not composed with a pencil and $h^0(L) < p_g(S)$, then
$$
(\sigma^*K_{S'})L \ge 3 h^0(L)- 5.
$$
\end{theorem}

\begin{proof}
If $|L|$ is composed with a pencil, we can write
$$
L \sim_{\rm num} rC + Z
$$
where $Z \ge 0$ is the fixed part, $C$ is a general member of the pencil and $r \ge h^0(L)-1$.
Let $C' = \sigma(C)$. Then $p_a(C') \ge 3$ by our assumption.

If $K_{S'}C' \ge 3$, then
$$
(\sigma^*K_{S'})L \ge r K_{S'}C' \ge 3h^0(L)-3.
$$
If $K_{S'}C' \le 2$, since $C'$ is nef and $C' \le K_S$, we get $C'^2 \le 2$. It implies that $K_{S'}C' = C'^2 =2$ as $p_a(C') \ge 3$. Note that $K^2_{S'} \ge K_{S'}C'$. By the Hodge index theorem, we have $K^2_{S'} = 2$, $K_{S'} \sim_{\rm lin} C'$ and $p_g(S) = p_g(S') = 2$. This is just the exceptional case.

When $\phi_L$ is generically finite, we can assume that $h^0(L) \ge 3$.
Since $h^0(L) < p_g(S)$,
by the Hodge index theorem and the Castelnuovo inequalities from both \cite[Th\'eor\`eme 5.5]{Be} and Theorem \ref{castelnuovo}, we have
\begin{eqnarray*}
((\sigma^*K_{S'})L)^2 & \ge & L^2 K^2_{S'} \ge (3h^0(L)-7)(3p_g(S)-7) \\
& \ge & 9 (h^0(L))^2 - 33h^0(L) + 28 \\
& > & 9 (h^0(L))^2 - 36h^0(L) + 36 \\
& = & (3h^0(L) - 6)^2.
\end{eqnarray*}
It implies that
$$
(\sigma^*K_{S'})L \ge 3h^0(L)-5.
$$
It ends the proof.
\end{proof}

We also give the following numerical result when $S$ has a free pencil.

\begin{theorem} \label{castelnuovoirr}
Suppose that $S$ has a free hyperelliptic pencil of genus $g \ge 6$. Then
$$
(\sigma^*K_{S'})L \ge 3h^0(L) - (4g-4).
$$
\end{theorem}

\begin{proof}
Denote by $C$ a general member of this pencil. Up to birational transformation, we can assume that $|L|$ is base point free.
In particular, $LC \le 2g-2$ should be an even number.

If $LC=0$, then $L \sim_{\rm num} aC$ where $a \ge h^0(L)-1$. Then
$$
(\sigma^*K_{S'})L= a (2g-2) > 3h^0(L) - 3.
$$

If $LC \ge 6$,
by \cite[Theorem 1.1]{YZ2}, one has
$$
L^2 \ge \frac {4LC}{LC+2} h^0(L) - 2 LC \ge 3h^0(L) - (4g-4).
$$

If $2 \le LC \le 4$, resume the notations in Theorem \ref{decomposition}, Proposition \ref{algcase1} and Lemma \ref{algsumai}. Denote
$L_0=L$ and $P=\sigma^*K_{S'}$. We have
\begin{eqnarray*}
h^0(L_0) & \le & \sum_{i=0}^N a_i r_i, \\
PL_0 & \ge & (a_0-1)d_0 + \sum_{i=1}^N a_id_i.
\end{eqnarray*}
Here $d_i=2g-2 \ge 10$ for all $i$ and $r_i \le h^0(L|_C) \le 3$ by the Clifford's inequality. It follows that
$$
d_i > 3r_i.
$$
Hence
$$
PL_0 \ge \sum_{i=0}^N a_id_i - d_0 > 3 \sum_{i=0}^N a_ir_i - d_0 \ge 3 h^0(L_0) - (2g-2).
$$
It finishes the proof.
\end{proof}

\section{Relative Noether inequalities}
In this section, we will prove several relative Noether inequalities. The relative Noether inequality
in \cite{Zh} studies linear series on fibered varieties over curves, while the relative
Noether inequalities in this section are devoted to studying fibered varieties whose fibers are curves.

\begin{assumption} \label{as1}
Throughout this section, we assume that:
\begin{itemize}
\item $X$ is a Gorenstein minimal projective 3-fold of general type;
\item $f: X \to Y$ is a fibration of curves of genus $g$ from $X$ to a normal projective surface $Y$;
\item $L$ is a nef and big line bundle on $Y$ such that $|L|$ is base point free. Write $B=f^*L$.
\item If $p_g(X) > 0$, write
$$
K_X=\sum_{i=1}^{I_0} H_i + V,
$$
where each $H_i$ is an irreducible and reduced horizontal divisor ($H_i$ and $H_j$ might be the same) and $V$ is the vertical part. Note that we have $I_0 \le 2g-2$.
Since $B|_{H_i}$ is big on $H_i$, we can find an integer $k>0$ such that
$(kB-K_X)|_{H_i}$ is pseudo-effective for each $i$. Also we can assume that $kB-V$ is pseudo-effective by increasing $k$.
\end{itemize}
\end{assumption}

\begin{theorem} \label{relnoether1}
Under Assumption \ref{as1}. Then one has
$$
h^0(K_X)-\frac {K^3_X} 4 \le \frac {3K^2_XB + 2K_XB^2} 4 + 2(2k+1).
$$
\end{theorem}

\begin{proof}
To prove this result, one can assume that $h^0(K_X) > 0$.

Choose two very general members in $|B|$ and denote by $\sigma: X_0 \to X$ the blow-up of their intersection. Let $F$ be its proper transform. Then we get a fibration $X_0 \to \mathbb P^1$ with general fiber $F$. Denote $L_0=\sigma^*K_X$, $P_0=\sigma^*(K_X+B)$ and $B_0=\sigma^*B$.

Now, apply Theorem \ref{decomposition} to $X_0$, $L_0$, $P_0$ and also use the notations in Proposition \ref{algcase1}.
We can get
\begin{eqnarray*}
h^0(L_0) & \le & \sum_{i=0}^N a_{i} r_{i}; \\
P_0L^2_0 & \ge & 2\sum_{i=0}^N a_{i} d_i - 2d_0.
\end{eqnarray*}
To compare $r_i$ and $d_i$, note that $P_0|_F=K_F$. By Theorem \ref{noether2},
we know that
$$
r_i \le \frac 12 d_i + 2.
$$
By Lemma \ref{algsumai}, we have
$$
\sum_{i=0}^N a_i \le \frac{P_0L^2_0}{d_0} - a_0 + 2 \le \frac{P_0L^2_0}{d_0} + 1.
$$
Combine the above inequalities and we get
$$
h^0(L_0) - \frac {P_0L^2_0} 4 \le 2 \sum_{i=0}^N a_i + \frac{d_0} 2 \le \frac {2 P_0L^2_0} {d_0} + 2 + \frac{d_0}{2}.
$$
On the other hand, note that $kB-V$ and $(kB-K_X)|_{H_i}$ are pseudo-effective. We can get
\begin{eqnarray*}
(K_X + B) K_X V & \le & k(K_X + B) K_X B, \\
\sum_{i=1}^{I_0} (K_X + B) K_X H_i & \le & k\sum_{i=1}^{I_0} (K_X + B) B H_i \le k (K_X + B) K_X B.
\end{eqnarray*}
As a result, we have
\begin{eqnarray*}
\frac {P_0L^2_0} {d_0} & = & \frac {K^2_X(K_X + B)}{(K_X + B)K_X B} \\
& = &  \frac {(K_X + B)K_X V}{(K_X + B)K_X B} + \sum_{i=1}^{I_0} \frac {(K_X + B)K_X H_i}{(K_X + B)K_X B} \le 2k.
\end{eqnarray*}
Therefore,
\begin{eqnarray*}
h^0(K_X) & \le & \frac{P_0L^2_0}{4} + \frac{2P_0L^2_0}{d_0} + \frac{d_0}2 + 2 \\
& = & \frac{K^3_X + K^2_X B}4 + 4k + \frac{K^2_X B + K_X B^2}2 + 2 \\
& = & \frac{K^3_X} 4 + \frac{3K^2_X B + 2K_XB^2}4 + 2(2k+1).
\end{eqnarray*}
It finishes the proof.
\end{proof}

\begin{theorem} \label{relnoether2}
Still under Assumption \ref{as1}. Suppose that $X$ has no hyperelliptic pencil. Then one has
$$
h^0(K_X)-\frac{K^3_X}{6} \le \frac{3K^2_X B + 2K_X B^2}6 + \frac 73 (2k+1).
$$
\end{theorem}

\begin{proof}
It is very similar to the previous proof. We sketch it here. Under the same setting as in the proof of Theorem \ref{relnoether1}, we still have
\begin{eqnarray*}
h^0(L_0) & \le & \sum_{i=0}^N a_{i} r_{i}; \\
P_0L^2_0 & \ge & 2\sum_{i=0}^N a_{i} d_i - 2d_0.
\end{eqnarray*}
Here, since $F$ has no hyperelliptic pencil, by Theorem \ref{castelnuovo2}, we have
$$
r_i \le \frac 13 d_i + \frac 73.
$$
Combine with Lemma \ref{algsumai}. It gives
$$
h^0(L_0) - \frac{P_0L^2_0} 6 \le \frac 73 \sum_{i=0}^N a_i + \frac {d_0}3 \le \frac {7P_0L^2_0}{3d_0} + \frac 73 + \frac{d_0}3.
$$
Hence
\begin{eqnarray*}
h^0(K_X) & \le & \frac{P_0L^2_0} 6 + \frac {7P_0L^2_0}{3d_0} + \frac 73 + \frac{d_0}3 \\
& \le & \frac{K^3_X + K^2_XB} 6 + \frac 73 (2k+1) + \frac{K^2_XB + K_XB^2} 3 \\
& = & \frac{K^3_X}{6} + \frac{3K^2_X B + 2K_X B^2}6 + \frac 73 (2k+1).
\end{eqnarray*}
It finishes the proof.
\end{proof}

\begin{theorem}\label{essential}
Still under Assumption \ref{as1}. Suppose that $f$ is hyperelliptic and $g \ge 6$. Then
$$
h^0(K_X)-\frac{K^3_X}{6} \le \frac{3K^2_X B + 2K_X B^2}6 + \frac {4g-4}3 (2k+1).
$$
\end{theorem}

\begin{proof}
Here we simply follow the proof of Theorem \ref{relativenoether3fold2}.
The only difference is that
in this case, $F$ has a free of genus $g$ induced by $f$.
By Theorem \ref{castelnuovoirr}, we have
$$
r_i \le \frac 13 d_i + \frac{4g-4}{3}.
$$
Follow the proof of Theorem \ref{relativenoether3fold2} almost verbatim. We can get
\begin{eqnarray*}
h^0(K_X) & \le & \frac{K^3_X + K^2_XB} 6 + \frac {4g-4} 3 (2k+1) + \frac{K^2_XB + K_XB^2} 3 \\
& = & \frac{K^3_X}{6} + \frac{3K^2_X B + 2K_X B^2}6 + \frac {4g-4}3 (2k+1).
\end{eqnarray*}
We leave the detailed proof to the interested readers.
\end{proof}

\section{Proof of Theorem \ref{3fold}}
In this section, we will prove Theorem \ref{3fold}. Note that the strategy here has been used by Pardini \cite{Pa} in the proof of Severi inequality for surfaces.

We first state a theorem that has been used in \cite{Zh}.
\begin{theorem} \label{albanese}
Let $X$ be a projective, minimal, normal and irregular variety. Denote by $a(X)$ its Albanese image. For each $d \in \mathbb N$, let $X_d$ be the $d$-th Albanese
lifting of $X$. Then for each $i=0, \cdots, \dim a(X)-1$, we have
$$
\lim_{d \to \infty} \frac{h^i(\mathcal O_{X_d})}{d^{2m}}=0.
$$
Here $m=h^1(\mathcal O_X)$.
\end{theorem}

\begin{proof}
Since the Albanese lifting is \'etale, $X_d$ is also minimal. Furthermore, $X_d$ has only terminal singularities, which are rational. Hence it suffices to assume that $X$ is smooth. In this case, the result is just \cite[Theorem 4.1]{Zh}. Also see \cite[Remark 1.4]{HK} for the generic vanishing theorem for singular varieties with rational singularities in characteristic $0$.
\end{proof}

Go back to the 3-fold case.
Now assume that $X$ has Albanese dimension two.
Let $X_d$ be the $d$-th Albanese lifting of $X$. We have the following diagram:
$$
\xymatrix{X_d \ar[r]^{\phi_d} \ar[d]_{\alpha_d} & X \ar[d]^{\alpha={\rm Alb}_X} \\
A \ar[r]^{\mu_d} & A}
$$
Let $m=h^1(\mathcal O_X)$. We have
\begin{eqnarray*}
h^0(K_{X_d}) & \ge & \chi(\omega_{X_d}) + h^2(\mathcal O_{X_d})- h^1(\mathcal O_{X_d}) + 1 \\
& \ge & d^{2m} \chi(\omega_X) - h^1(\mathcal O_{X_d}).
\end{eqnarray*}
By Theorem \ref{albanese}, it follows that
$$
h^0(K_{X_d}) \ge d^{2m} \chi(\omega_X) + o(d^{2m}).
$$
Note that in order to prove Theorem \ref{3fold}, we can assume that $\chi(\omega_X)>0$. Thus one can find an integer $d>0$ such that $h^0(K_{X_d}) > 0$. Also note that Theorem \ref{3fold} is true up to \'etale covers. Without loss of generality, let us assume that $h^0(K_X) > 0$.

Let $X \xrightarrow{g_0} Y \xrightarrow{h_0} A$ (resp. $X_d \xrightarrow{g_d} Y_d \xrightarrow{h_d} A$) be the Stein factorization of $\alpha$ (resp. $\alpha_d$).
Let $H$ be a sufficiently ample line bundle on $A$ and $L_d=h^*_d H$ for all $d$. Write
$B_d=g^*_d L_d$ for all $d$.
Then we have \cite[Chapter 2, Proposition 3.5]{BL}
$$
d^2 B_d \sim_{\rm num} \phi^*_d B_0.
$$

Resume the notations in previous section. We can write
$$
K_{X_d} = \phi^*_d K_X =\sum_{i=1}^{I_0} \phi^*_d H_i + \phi^*_d V.
$$
By pulling back from $X$ to $X_d$, one can check that
$$
(kg^*_d B_0 - K_{X_d})|_{\phi^*_d H_i}, \quad k \phi^*_d B_0 - \phi^*_d V
$$
are both pseudo-effective. Apply the above numerical equivalence and we get
$$
(k d^2  B_d - K_{X_d})|_{\phi^*_d H_i}, \quad k d^2 B_d - \phi^*_d V
$$
are both pseudo-effective. Now by Theorem \ref{relnoether1},
\begin{eqnarray*}
h^0(K_{X_d}) & \le & \frac {K^3_{X_d}}4 + \frac {3K^2_{X_d}B_d + 2 K_{X_d} B^2_d}4 + 2(2k d^2 + 1) \\
& = & \frac{d^{2m} K^2_X}4 + \frac{3d^{2m-2}K^2_X B_0 + 2d^{2m-4}K_X B^2_0} 4 + 2(2k d^2 + 1).
\end{eqnarray*}
Let $d \to \infty$. We can prove that $K^3_X \ge 4 \chi(\omega_X)$.

When the Albanese fiber is hyperelliptic of genus $g \ge 6$, using the same approach as above and by Theorem \ref{essential}, we can prove
$$
h^0(K_{X_d}) \le \frac{d^{2m} K^2_X}6 + \frac{3d^{2m-2}K^2_X B_0 + 2d^{2m-4}K_X B^2_0} 6 + \frac{4g-4}3 (2k d^2 + 1).
$$
Let $d \to \infty$ and we will have $K^3_X \ge 6 \chi(\omega_X)$.

Now we consider the case when the Albanese fiber is not hyperelliptic.
Note that $|L_d|$ is base point free on $Y_d$. Hence by the above numerical equivalence, we have
$$
d^{2m-4}L^2_0 = L^2_d \ge h^0(L_d) - 2 = h^0(B_d) - 2,
$$
i.e., $h^0(B_d) \sim O(d^{2m-4})$. Therefore,
up to \'etale covers, we can assume that $h^0(K_X) > h^0(B_0)$. In particular, it implies that the Albanese fibration factors through the canonical map of $X$.

We claim that \textit{$X$ can not have hyperelliptic pencils}.
Otherwise, suppose there is a hyperelliptic pencil on $X$. We would have the following two possibilities: either this pencil is contracted by $\phi_{K_X}$ or not. If it is contracted by $\phi_{K_X}$, it would also be contracted by the Albanese map, which is impossible since the Albanese pencil is nonhyperelliptic. The second case is still impossible because if this pencil is not contracted by $\phi_{K_X}$, then its image under $\phi_{K_X}$ is a $\mathbb P^1$ pencil. Since the Albanese image of $X$ can not have any $\mathbb P^1$ pencil, by the factorization of the Albanese map, this pencil has to be contracted by the Albanese map. It contradicts with our assumption again.

Similar to the above claim, we can prove that $X_d$ has no hyperelliptic pencil. Thus by Theorem \ref{relnoether2}, we have
$$
h^0(K_{X_d}) \le \frac{d^{2m} K^2_X}6 + \frac{3d^{2m-2}K^2_X B_0 + 2d^{2m-4}K_X B^2_0} 6 + \frac 73 (2k d^2 + 1).
$$
Similar as before, we can get $K^3_X \ge 6 \chi(\omega_X)$ by letting $d \to \infty$.

\begin{remark}
In fact, if $X$ is $\mathbb Q$-Gorenstein, the same approach still works. The only difference is that in this case $K_X$ is a $\mathbb Q$-linear combination of prime divisors. This does not affect the whole proof. However, if $X$ is $\mathbb Q$-Gorenstein but not Gorenstein,
$\chi(\omega_X)$ might be negative.
\end{remark}

\begin{example} \label{example}
Fix an elliptic curve $E$. Take two curves $C_i$ $(i=1, 2)$ with genus $g_i \ge 2$ and involutions $\tau_i$ such that
$C_1 / \langle \tau_1 \rangle = E$ and $C_2 / \langle \tau_2 \rangle = \mathbb P^1$. Let the 3-fold $X$ be the quotient
of $C_1 \times C_1 \times C_2$ by the diagonal involution. Hence $X$ is $\mathbb Q$-Gorenstein but not Gorenstein. Let
$\alpha: X \to E \times E \times \mathbb P^1$ be the induced cover. Then
\begin{eqnarray*}
\alpha_* \mathcal O_X & = & \mathcal O_E \boxtimes \mathcal O_E \boxtimes \mathcal O_{\mathbb P^1} \oplus
L_1 \boxtimes L_1 \boxtimes \mathcal O_{\mathbb P^1}  \oplus \\
& & \mathcal O_E \boxtimes L_1 \boxtimes L_2 \oplus
L_1 \boxtimes \mathcal O_E \boxtimes L_2
\end{eqnarray*}
where $L_1$ and $L_2$ are determined by the branch locus of $C_1 \to E$ and $C_2 \to \mathbb P^1$.
It is easy to see that $\chi(\mathcal O_X) > 0$. Hence $\chi(\omega_X) < 0$. It is also easy to see that
$X$ has Albanese dimension two.
\end{example}

\section{Proof of Theorem \ref{alb1}}
In this section, we will prove Theorem \ref{alb1}.
We have the following relative Noether inequality for fibered 3-folds.

\begin{theorem} \label{relativenoether3fold1}
Let $X$ be a Gorenstein minimal 3-fold of general type, $Y$ a smooth curve and
$f: X \to Y$ is a fibration with a smooth general fiber $F$. Then
$$
h^0(\omega_{X/Y}) \le \left(\frac 14 + \frac{1}{K^2_F} \right) \omega^3_{X/Y} + \frac {K^2_F + 4}{2}
$$
except for the case when $K^2_F=1$, $p_g(F)=2$ and $q(F)=0$.
\end{theorem}

\begin{proof}
We know in this case, $\omega_{X/Y}$ is nef. Resume the notations in Theorem \ref{decomposition} and Proposition \ref{algcase1}. Denote
$L_0=P=\omega_{X/Y}$. It follows that
\begin{eqnarray*}
h^0(L_0) & \le & \sum_{i=0}^N a_{i} r_{i}; \\
L^3_0 & \ge & 2a_id_i + \sum_{i=1}^N a_{i} (d_{i-1}+d_i) - 2d_0.
\end{eqnarray*}
Hence
$$
h^0(L_0) - \frac{L^3_0}{4} \le \frac{d_0}2 + \left(r_0-\frac 12 d_0\right)a_0 + \sum_{i=1}^N \left(r_i-\frac 12 d_i-\frac{d_{i-1}-d_i}4\right) a_i.
$$

By Theorem \ref{noether}, \ref{noether2} and Remark \ref{r0r1}, we can always get
$$
r_0 \le \frac 12 d_0 + 2, \quad r_i \le \frac 12 d_i + \frac{d_{i-1}-d_i}4 + 1 \quad (i>0)
$$
except when $r_0=r_1$ and $d_1 \le 2r_1 - 3$.

If we are not in the exceptional case, then
$$
h^0(L_0) - \frac{L^3_0}{4} \le \frac{d_0}2 + 2a_0 + \sum_{i=1}^N a_i.
$$
By Lemma \ref{algsumai},
$$
2 a_0 + \sum_{i=1}^N a_i \le \frac{L^3_0}{d_0} + 2 \le \frac{L^3_0}{d_0} + 2.
$$
Hence
$$
h^0(L_0) \le \left(\frac 14 + \frac{1}{d_0} \right) L^3_0 + \frac {d_0 + 4}{2}.
$$

Finally, let us consider the exceptional case. Recall that we have $r_0=r_1$ and $d_1 \le 2r_1 - 3$.  By Theorem \ref{noether2}, $|L_1|_{F_1}|$ defines a generically finite morphism. Thus from Theorem \ref{noether}, we know that
$$
(L_1|_{F_1})^2 \ge 2r_1 - 4.
$$
Note that $\pi^*_0 K_F > L_1|_{F_1}$. By the Hodge index theorem,
$$
(2r_1 - 3)^2 \ge d^2_1 = ((\pi^*_0 K_F)(L_1|_{F_1}))^2 > (L_1|_{F_1})^2 K^2_F,
$$
which implies that $K^2_F \le 2r_0 - 2$.
On the other hand, $|K_F|$ is not base point free, thus $K^2_F \ge 2r_0 - 3$. It also implies that $d_1 = 2r_1 - 3$ and $(L_1|_{F_1})^2 = 2r_1 - 4$. Therefore, we have two possibilities:
\begin{itemize}
\item [(1)] $d_1=2r_1-3$ and $K^2_F = d_1$;
\item [(2)] $d_1=2r_1-3$ and $K^2_F = d_1+1$.
\end{itemize}

In Case (1), choose a blow up $\pi: X' \to X$ such that the movable part $|M|$ of $|\pi^* \omega_{X/Y}|$ is base point free. Write $F'=\pi^*F$. Abusing the notation, we denote the new $L_0=M$ and $P=\pi^* \omega_{X/Y}$. Then we will have a new sequence of $r_i$'s and $d_i$'s. Under this new setting, $r_{i-1}>r_i$ for each $i>0$. Running the same process as before, we can get
$$
h^0(\omega_{X/Y}) = h^0(M) \le \left(\frac 14 + \frac 1{d_{M}}\right)\omega^3_{X/Y} + \frac{d_{M}+4}{2},
$$
where $d_{M} = (\pi^*K_F) (M|_{F'})$. We need to show that $d_{M} = K^2_F$. In fact, $d_{M} = K_F (\pi(M)|_F) \le K^2_F$. On the other hand, since $L_1$ comes from the movable part of $|\omega_{X/Y} - a_0 F|$, it implies that $\pi_0 (L_1) \le \pi(M)$. In particular, it means that
$$
K^2_F=d_1 = K_F (\pi_0 (L_1)|_F) \le K_F (\pi(M)|_F) = d_{M}.
$$
Hence $d_{M} = K^2_F$.

We claim that Case (2) does not occur. Write
$$
|K_F| = |V| + Z,
$$
where $|V|$ is the movable part of $|K_F|$ and $Z \ge 0$. Since $r_1=r_0$, we see that $\pi_0(L_1)|_{F} = V$ and so $d_1 = K_F V$. Since $K^2_F > d_1$, it means that $K_FZ>0$ and $Z > 0$. By the 2-connectedness of the canonical divisor, we have $VZ \ge 2$. Note that $V^2 \ge 2r_0-4$. It implies that
$$
2r_0 - 2 = K^2_F = V^2 + VZ + K_FZ \ge 2 r_0 - 2 + K_FZ,
$$
i.e., $K_FZ = 0$ and $K_FV = K^2_F$. It contradicts with $K^2_F > d_1$.
\end{proof}

\begin{remark}
Recall in \cite{YZ2} that if $f: S \to C$ is a relative minimal fibered surface of genus $g \ge 2$. One has
$$
h^0(\omega_{S/C}) \le \frac{g}{4g-4}\omega^2_{S/C}+g = \left(\frac 14 + \frac 1{2\deg \omega_F} \right) \omega^2_{S/C} + \frac{\deg \omega_F + 2}{2}.
$$
Theorem \ref{relativenoether3fold1} is a natural generalization of the above result. Moreover, based on several results about linear series on surfaces in positive characteristic (c.f. \cite{Li1, Li2}), our method can be also used study the fibered 3-folds in positive characteristic.

It is also interesting to compare Theorem \ref{relativenoether3fold1} with the slope inequalities proved by Ohno \cite{Oh} and Barja \cite{Ba1} for fibered 3-folds over curves. In their papers, Xiao's method on the Harder-Narasimhan filtration \cite{Xi} plays a very important role and it works only in characteristic zero.

\end{remark}

By the same technique, we can also show the following Noether type inequality for fibered 3-folds over curves in order for independent interests.
\begin{theorem}
Let $X$ be a Gorenstein and minimal fibered 3-fold of general type fibered over a smooth curve $Y$
with a smooth general fiber $F$. Then
$$
p_g(X) \le \left(\frac 14 + \frac{1}{K^2_F} \right) K^3_{X} + \frac {K^2_F + 4}{2}
$$
except for the case when $K^2_F=1$, $p_g(F)=2$ and $q(F)=0$.
\end{theorem}

\begin{proof}
We only need to replace $\omega_{X/Y}$ in the proof of Theorem \ref{relativenoether3fold1} by $K_X$.
\end{proof}

\textit{From now on}, we assume that $f: X \to Y$ is the induced fibration by the Albanese map from $X$ to a projective curve $Y$ with a smooth
general fiber $F$. One has $g(Y)=h^1(\mathcal O_X)$. We put the following remark before the proof.
\begin{remark} \label{observ}
Suppose we can prove that
$$
K^3_X \ge a_F \chi(\omega_X) + b_F,
$$
where $a_F$ and $b_F$ only depend on the numerical invariants of $F$. Then we can get
$$
K^3_X \ge a_F \chi(\omega_X).
$$
This philosophy has been applied by Bombieri and Horikawa in \cite{Bo, Ho} and also \cite{YZ2}. In fact, since $g(Y)=h^1(\mathcal O_X) > 0$.
Using the degree $d$ \'etale base change $\pi: Y' \to Y$, we can get a new fibration $f': X' \to Y'$, where $X'=X \times_Y Y'$.
We still have
$$
K^3_{X'} \ge a_F \chi(\omega_{X'}) + b_F.
$$
Note that
$$
K^3_{X'} = d K^3_X, \quad \chi(\omega_{X'})=d \chi(\omega_X).
$$
The conclusion will follow after we let $d \to \infty$. This remark will also be used in the next section.
\end{remark}

\begin{prop} \label{pg0}
If $p_g(F)=0$, then
$$
K^3_X \ge 6 \chi(\omega_X).
$$
\end{prop}

\begin{proof}
In this situation, $p_g(X)=0$. From the nefness of $\omega_{X/Y}$, we get $\omega^3_{X/Y} \ge 0$. Since $\omega_{X/Y}=K_X - (2g(Y)-2)F$, we get
$$
K^3_X \ge 6(g(Y)-1) K^2_F \ge 6(g(Y)-1) \ge 6 \chi(\omega_X).
$$
It finishes the proof.
\end{proof}

\begin{prop} \label{rough2}
If $p_g(F) > 0$ and $(p_g(F), K^2_F) \ne (2, 1)$, then
$$
\chi(\omega_X) \le  \left(\frac 14 + \frac{1}{K^2_F} \right) K^3_{X}.
$$
\end{prop}

\begin{proof}
From Theorem \ref{relativenoether3fold1}, we get
\begin{eqnarray*}
h^0(\omega_{X/Y}) & \le & \left(\frac 14 + \frac{1}{K^2_F} \right) (K^3_{X} - 6K^2_F(g(Y)-1)) + \frac {K^2_F + 4}{2} \\
& = & \left(\frac 14 + \frac{1}{K^2_F} \right) K^3_{X} - 6\left(\frac {K^2_F} 4 + 1 \right) (g(Y)-1) + \frac {K^2_F + 4}{2}.
\end{eqnarray*}

On the other hand, by \cite{Ko1, Ko2}, $f_*\omega_{X/Y}$ and $R^1 f_*\omega_{X/Y}$ are both semipositive. Following from \cite[Lemma 2.4, 2.5]{Oh}, we have
\begin{eqnarray*}
h^0(\omega_{X/Y}) & \ge & \deg f_*\omega_{X/Y} - p_g(F)(g(Y)-1) \\
& \ge & \deg f_*\omega_{X/Y} - \deg R^1 f_*\omega_{X/Y} - p_g(F)(g(Y)-1) \\
& = & \chi(\mathcal \omega_X) - (\chi(\mathcal O_F)+p_g(F)) (g(Y)-1).
\end{eqnarray*}
By applying Remark \ref{observ}, to prove the conclusion, it suffices to prove that
$$
6 \left(\frac {K^2_F} 4 + 1 \right) \ge \chi(\mathcal O_F)+p_g(F).
$$
It is easy to see that this inequality follows from the classical Noether inequality, since
$$
6 \left(\frac {K^2_F} 4 + 1 \right) > K^2_F + 6 \ge 2p_g(F) + 2 > \chi(\mathcal O_F)+p_g(F).
$$
It finishes the proof.
\end{proof}

From the above proposition, we see that Theorem \ref{alb1} holds if $K^2_F \ge 4$. In the following, we will consider the case when $K^2_F \le 3$.

\begin{prop} \label{pg23}
We have
$$
K^3_X \ge 2 \chi(\omega_X)
$$
in the following two cases:
\begin{itemize}
\item [(1)] $p_g(F)=2$, $K^2_F=2, 3$;
\item [(2)] $p_g(F)=3$, $K^2_F =3$.
\end{itemize}
\end{prop}

\begin{proof}
We first prove that
$$
\omega^3_{X/Y} \ge 2h^0(\omega_{X/Y})-6.
$$
Suppose we have proven the above result. As before, we still have
\begin{eqnarray*}
h^0(\omega_{X/Y}) & \ge & \chi(\omega_X) - (\chi(\mathcal O_F) + p_g(F))(g(Y)-1), \\
\omega^3_{X/Y} & = & K^3_X - 6K^2_F (g(Y)-1).
\end{eqnarray*}
It is easy to check that
$$
6K^2_F \ge 2(2p_g(F) + 1) \ge 2(\chi(\mathcal O_F) + p_g(F))
$$
in these cases. Hence it will imply that $K^3_X \ge 2\chi(\omega_X)$ by Remark \ref{observ}.

To prove that $\omega^3_{X/Y} \ge 2h^0(\omega_{X/Y})-6$, we can assume that $h^0(\omega_{X/Y}) \ge 4$. Hence we have
the relative canonical map $\phi_{\omega_{X/Y}}$. Note that $p_g(F) \le 3$ and we have
$$
0 \longrightarrow H^0(\omega_{X/Y}(-F)) \longrightarrow H^0(\omega_{X/Y}) \longrightarrow H^0(K_F).
$$
Thus $h^0(\omega_{X/Y}(-F))>0$ and it implies that $f$ factors through $\phi_{\omega_{X/Y}}$.

Choose a blow-up $\pi: X' \to X$ such that
the movable part $|M|$ of $|\pi^*\omega_{X/Y}|$ is base point free. Write $F'=\pi^*F$ and
$S$ as a general member of $|M|$.

If $\dim \phi_{\omega_{X/Y}}(X)=1$, then
$$
M \sim_{\rm num} aF' + Z
$$
where $Z \ge 0$ and $a \ge h^0(\omega_{X/Y})$. Hence
$$
\omega^3_{X/Y} \ge a(\pi^*\omega_{X/Y})^2 F' = a K^2_F \ge 2 h^0(\omega_{X/Y}).
$$

If $\dim \phi_{\omega_{X/Y}}(X)=2$, denote by $C'$ a general member of the induced pencil by $\phi_{M}$.
Then $M|_S$ is a free pencil and
$$
M|_S \sim_{\rm num} a C',
$$
where $a \ge h^0(\omega_{X/Y}) - 2$. Hence
$$
\omega^3_{X/Y} \ge ((\pi^*\omega_{X/Y})C') (h^0(\omega_{X/Y}) - 2).
$$
On the other hand, since $f$ factors through $\phi_{\omega_{X/Y}}$, we can find a general $F'$ such
that $C' \subset F'$. Then $(\pi^*\omega_{X/Y})C'=K_FC$, where $C=\pi(C')$. Note that $g(C) \ge 2$ and $K^2_F \ge 2$,
by the Hodge index theorem, we can get $K_FC \ge 2$. So
$$
\omega^3_{X/Y} \ge 2(h^0(\omega_{X/Y}) - 2).
$$

If $\dim \phi_{\omega_{X/Y}}(X)=3$, then $\dim \phi_{M}(S)=2$. By Theorem \ref{noether},
$(M|_S)^2 \ge 2h^0(M|_S) - 4$, which implies
$$
\omega^3_{X/Y} \ge M^3 \ge 2h^0(M|_S) - 4 \ge 2h^0(\omega_{X/Y}) - 6.
$$
It finishes the proof.
\end{proof}
In fact, the above method has been applied by Chen \cite{Ch1} for the study of the canonical linear system. Here we use this method for the relative canonical linear system.

\begin{prop} \label{pg32}
If $p_g(F)=3$ and $K^2_F = 2$, then
$$
K^3_X \ge \frac{12}{7} \chi(\omega_X).
$$
\end{prop}

\begin{proof}
In this case, $h^1(\mathcal O_F)=0$ and $\chi(\mathcal O_F) = 4$. Apply the same method as in Proposition \ref{pg23}. We can still get
\begin{eqnarray*}
\omega^3_{X/Y} & \ge & 2h^0(\omega_{X/Y})-6, \\
h^0(\omega_{X/Y}) & \ge & \chi(\omega_X) - 7(g(Y)-1), \\
\omega^3_{X/Y} & = & K^3_X - 12 (g(Y)-1).
\end{eqnarray*}
Then the result follows from Remark \ref{observ}.
\end{proof}

\begin{remark}
In fact, from the above inequalities, we can also get
$$
K^3_X \ge 2 \chi(\omega_X) - 2 h^1(\mathcal O_X)-4.
$$
But here $h^1(\mathcal O_X)$ is still involved.
\end{remark}

\begin{prop} \label{pg21}
If $p_g(F)=2$ and $K^2_F = 1$, then
$$
K^3_X \ge \frac 43 \chi(\omega_X).
$$
\end{prop}

\begin{proof}
By a very recent result of Hu (c.f. \cite{Hu}), we know that in this case,
$$
K^3_X \ge \frac 43 \chi(\omega_X) - 2.
$$
Therefore, the result follows from Remark \ref{observ}.
\end{proof}

\begin{prop} \label{pg1}
If $p_g(F)=1$ and $K^2_F > 1$, then
$$
K^3_X \ge 2 \chi(\omega_X).
$$
\end{prop}

\begin{proof}
In this case, we know that $p_g(X) > 0$. First, let us assume that $p_g(X) \ge 2$. Then
the canonical maps of $X$ will factor through $f$. So we can write
$$
K_X \sim_{\rm num} rF + Z,
$$
where $r \ge p_g(X)$ and $Z \ge 0$. It follows that
\begin{eqnarray*}
K^3_X & = & K^2_X (\omega_{X/Y}+(2g(Y)-2)F) = (2g(Y)-2)K^2_F + \omega_{X/Y}K^2_X \\
& \ge & (2g(Y)-2)K^2_F + r\omega_{X/Y}K_X F = (r+2g(Y)-2)K^2_F\\
& \ge &2 \chi(\omega_X).
\end{eqnarray*}

Second, if $p_g(X)=1$ and $h^1(\mathcal O_X)>1$, then $\chi(\omega_X) \le h^1(\mathcal O_X)$. We have
$$
K^3_X \ge 6(g(Y)-1)K^2_F \ge 2\chi(\omega_X).
$$

The only missing case is when $p_g(X)=1$ and $h^1(\mathcal O_X)=1$. In this case, $\chi(\omega_X)=1$ and $g(Y)=1$. Now
let $\mu: Y \to Y$ be any nontrivial \'etale map and $X'=X \times_{\mu} Y$. We have $\chi(\omega_{X'})>1$.
So either $p_g(X') \ge 2$ or $h^1(\mathcal O_{X'}) \ge 2$. If $X'$ has Albanese dimension $\ge 2$, then by Theorem \ref{3fold},
$$
K^3_{X'} \ge 4\chi(\omega_{X'}).
$$
If not, we are just in one of the first two cases and
$$
K^3_{X'} \ge 2\chi(\omega_{X'}).
$$
In any situation, it will imply that $K^3_{X} \ge 2\chi(\omega_{X})$.
\end{proof}

\begin{prop} \label{pg11}
If $p_g(F)=1$ and $K^2_F = 1$, then we have
$$
K^3_X \ge 2 \chi(\omega_X).
$$
\end{prop}

\begin{proof}
Here we prove this result by studying the linear system $|2K_X|$.

Since $p_g(F)=1$ and $K^2_F = 1$, then $h^1(\mathcal O_F)=0$, $h^0(2K_F)=3$ and $|2K_F|$ is base point free (see \cite{Fr}).
Hence $\phi_{2K_F}$ is a generically finite morphism of degree $4$.

As before, we choose a blow-up $\pi: X' \to X$ such that the movable part $|M|$ of $|\pi^* (2K_X)|$ is base point free. Note that by the plurigenus formula of
Reid,
$$
h^0(M)=h^0(2K_X) \ge \frac 12 K^3_X - 3\chi(\mathcal O_X).
$$

Denote $F'=\pi^*F$. Consider
the following restriction map:
$$
\textrm{res}: H^0(X', M) \to H^0(F', M|_{F'}).
$$
Denote by $r$ the dimension of its image. So $1 \le r \le 3$.

If $r=1$, then $\phi_{M}(X')$ is a curve and $\phi_{M}$ factors through the fibration $X' \to Y$. In this
case,
$$
M \sim_{\rm num} a F' + Z
$$
with $a \ge h^0(M)$. Hence we have
$$
2 K^3_X  \ge M (\pi^*K_X)^2 \ge a K^2_F \ge \frac 12 K^3_X - 3\chi(\mathcal O_X),
$$
i.e., $K^3 \ge 2 \chi(\omega_X)$.

If $r=2$, write $L_0=M$ and $P=\pi^* (2K_X)$. Resume the notations in Theorem \ref{decomposition} and Proposition \ref{algcase1}. Similar to Theorem \ref{relativenoether3fold1}, we have
$$
h^0(L_0) - \frac{PL^2_0}{4} \le \frac{d_0}2 + \left(r_0-\frac 12 d_0\right)a_0 + \sum_{i=1}^N \left(r_i-\frac 12 d_i-\frac{d_{i-1}-d_i}4\right) a_i.
$$
Here $r_0=h^0(M|_{F'})$ and $d_0=2(\pi^* K_F)(M|_{F'})$. Also, by Proposition \ref{newr},
the above inequality still holds if we change $r_0$ by $r$.

We claim that $d_0=4$ in this case. In fact, we know that $d_0 \le 4K^2_F =4$.
If $r_0 =3$, then $M|_{F'} = \pi^*(2K_F)$ and $d_0=4K^2_F=4$. If not, then $r_0=r=2$.
By \cite[Lemma 2.5]{Ch1},
we know that $(\pi^*K_F) (M|_{F'}) \ge 2$, which still gives $d_0 \ge 4$. Hence the claim is true and we have
$$
r-\frac 12 d_0 \le 0.
$$

For $i>0$, we know that $r_i < r_0 = 3$ by Remark \ref{r0r1}. Moreover, by \cite[Lemma 2.5]{Ch1} again, $d_i \ge 4$ if $r_i=2$, and $r_{i-1} \ge 2$ if $r_i=1$, which implies $d_{i-1} \ge 4$. From this, one can check that for any $i>0$,
$$
r_i-\frac 12 d_i-\frac{d_{i-1}-d_i}4 \le 0.
$$
Therefore, we have
$$
h^0(2K_X)-2K^3_X \le h^0(L_0) - \frac{PL^2_0}{4} \le 2.
$$

If $r=3$, then $\phi_{M}|_F = \phi_{2K_F}$. Thus $\phi_{M}$ is generically finite of degree $4$. Choose a general member $S \in |M|$. We have
$$
(M|_S)^2 \ge 4(h^0(M|_S)-2).
$$
Hence
$$
K^3_X \ge \frac 18 M^3 \ge \frac 12 (h^0(M|_S)-2) \ge \frac 12 (h^0(2K_X)-3).
$$

As a result, if $r \ge 2$, we always have
$$
h^0(2K_X) - 2K^3_X \le 3.
$$
Apply the plurigenus formula and Remark \ref{observ} to the above two cases. We can finish the whole proof.
\end{proof}

Now the proof of Theorem \ref{alb1} is straightforward.
\begin{proof}[Proof of Theorem \ref{alb1}]
From Proposition \ref{rough2}, we know that $K^2_X \ge 2\chi(\omega_X)$ holds when $K^2_F \ge 4$ and
$K^2_X \ge 3\chi(\omega_X)$ holds when $K^2_F \ge 12$ . If $K^2_F \le 4$, by Noether inequality,
$p_g(F) \le 3$. In this case, the result just comes from Proposition \ref{pg0}, \ref{pg23}, \ref{pg32}, \ref{pg21}, \ref{pg1}, \ref{pg11}.
\end{proof}


\section{Proof of Theorem \ref{alb12}}
In this section, we give the proof of Theorem \ref{alb12}.
We first give a better version of the relative Noether inequality for fibered 3-folds over curves.

\begin{theorem} \label{relativenoether3fold2}
Let $X$ be a Gorenstein minimal 3-fold of general type, $Y$ a smooth curve and $f: X \to Y$ is a fibration with a smooth general fiber $F$.
Suppose that $F$ has no hyperelliptic pencil and $(p_g(F), K^2_F) \ne (2, 2)$. Then
$$
h^0(\omega_{X/Y}) \le \left(\frac 16 + \frac{3}{2K^2_F} \right) \omega^3_{X/Y} + \frac {K^2_F + 7}{3}.
$$
\end{theorem}

\begin{proof}
Still resume the notations in Theorem \ref{decomposition} and Proposition \ref{algcase1}. Denote
$L_0=P=\omega_{X/Y}$. It follows that
\begin{eqnarray*}
h^0(L_0) & \le & \sum_{i=0}^N a_{i} r_{i}; \\
L^3_0 & \ge & 2a_0d_0 + \sum_{i=1}^N a_{i} (d_{i-1}+d_i) - 2d_0.
\end{eqnarray*}
Hence
$$
h^0(L_0) - \frac{L^3_0}{6} \le \frac{d_0}3 + \left(r_0-\frac 13 d_0\right)a_0 + \sum_{i=1}^N \left(r_i-\frac 13 d_i-\frac{d_{i-1}-d_i}6\right) a_i.
$$

By the Castelnuovo inequality, we can always get
$$
r_0 \le \frac 13 d_0 + \frac 73.
$$
Recall that by Remark \ref{r0r1}, $r_0 \ge r_1 > \cdots > r_N$.
We will prove in Lemma \ref{est} that if $r_0 > r_{1}$, then
$$
r_i \le \frac 13 d_i + \frac{d_{i-1}-d_i}6 + \frac 32, \quad (i>0).
$$
Assume the above result for now. It follows that
$$
h^0(L_0) - \frac{L^3_0}{6} \le \frac{d_0}3 + \frac 73 a_0 + \frac 32 \sum_{i=1}^N a_i.
$$
By Lemma \ref{algsumai},
$$
\frac 73 a_0 + \frac 32 \sum_{i=1}^N a_i \le \frac{3L^3_0}{2d_0} + 3 - \frac 23a_0 \le \frac{3L^3_0}{2d_0} + \frac 73.
$$
Hence
$$
h^0(L_0) \le \left(\frac 16 + \frac{3}{2d_0} \right) L^3_0 + \frac {d_0 + 7}{3}.
$$

Now we assume that $r_0=r_1$. It implies that $|K_F|$ has base locus by Remark \ref{r0r1}. Hence by the Castelnuovo inequality,
$$
K^2_F \ge 3 p_g(F) - 6.
$$
We have three exceptional cases.

\textbf{Case 1}. Suppose $K^2_F \ge 3 p_g(F) - 4$. We claim we still have
$$
r_i \le \frac 13 d_i + \frac{d_{i-1}-d_i}6 + \frac 32, \quad (i>0).
$$

This claim is true for $i \ge 2$ by Lemma \ref{est}. We only need
prove it for $i=1$.

If $|K_F|$ is composed with a pencil, by
Theorem \ref{castelnuovo2}, we have
$$
r_1 \le \frac 13 d_1 + 1.
$$
Hence the claim holds.

If $\phi_{K_F}$ is generically finite, then by Theorem \ref{castelnuovo},
$$
(L_1|_{F_1})^2 \ge 3 r_1 - 7.
$$
Therefore, by the Hodge index theorem,
$d_1 \ge \sqrt{K^2_F (L_1|_{F_1})^2}$,
which gives us
$$
r_1 \le \frac 13 d_1 + \frac 53.
$$
By our assumptions on $K^2_F$, we know that $d_0 \ge 3r_0 - 4 \ge d_1 + 1$. One can directly check that the claim holds in this case.

\textbf{Case 2}. Suppose $K^2_F \le 3 p_g(F) - 5$ and $|K_F|$ has only isolated base points.
Note that in this case, since $h^0(L_1|_{F_1})=h^0(K_F)$, $|L_1|_{F_1}|$ is just the proper transform of $|K_F|$. Hence
$$
(\pi^*_0 K_F)(L_1|_{F_1})=K^2_F.
$$

Again, choose a blow-up $\pi: X' \to X$ such that the movable part $|M|$ of $|\pi^* \omega_{X/Y}|$ is base point free. Let $F'=\pi^* F$.
Denote the new $L_0=M$ and $P=\pi^* \omega_{X/Y}$. We will have a new sequence of $r_i$'s and $d_i$'s. Under this new setting, for each $i>0$, we have $r_{i-1} > r_i$. By Lemma \ref{est},
$$
r_i \le \frac 13 d_i + \frac{d_{i-1}-d_i}6 + \frac 32, \quad (i>0).
$$
in this new setting. Run the same process as the non-exceptional case and we will have
$$
h^0(\omega_{X/Y}) \le \left(\frac 16 + \frac{3}{2d_{M}} \right) \omega^3_{X/Y} + \frac {d_{M} + 7}{3},
$$
where $d_{M} = (\pi^*K_F)(M|_{F'})$. We only need to show that $d_{M}=K^2_F$. This is quite similar to Theorem \ref{relativenoether3fold1}. In fact, we have $\pi_0(L_1) \le \pi(M)$ by the same reason as in Theorem \ref{relativenoether3fold1}. It implies that $|M|_{F'}|$ is the proper transform of $|K_F|$. By our assumption, $|K_F|$ has no fixed part. Thus $|\pi(M)|_F|  = |K_F|$ and $d_{M} = K^2_F$.

\textbf{Case 3}. Suppose $K^2_F \le 3 p_g(F) - 5$ and $|K_F|$ has a fixed part, i.e.,
$$
|K_F| = |V| + Z
$$
where $Z>0$. In this case, $|L_1|_{F_1}|$ is the proper transform of $|V|$.

Let $\pi: X' \to X$ and $|M|$ be the same as in Case 2. We still have
$$
h^0(\omega_{X/Y}) \le \left(\frac 16 + \frac{3}{2d_{M}} \right) \omega^3_{X/Y} + \frac {d_{M} + 7}{3},
$$
where $d_{M} = (\pi^*K_F)(M|_{F'})$. Using the similar argument, we can prove that $|\pi(M)|_F| = |V|$. As before, we only need to show that $K_FV = K^2_F$.

By the 2-connectedness of the canonical divisor, $VZ \ge 2$. Note that $V^2 \ge 3p_g(F)-7$ by Theorem \ref{castelnuovo}. Thus
$$
3 p_g(F) - 5 \ge K^2_F \ge K_F V = V^2 + VZ \ge 3 p_g(F)-5.
$$
It implies that $K_F V = 3 p_g(F)-5=K^2_F$.

It finishes the proof.
\end{proof}

As is said before, we need to show the following lemma.
\begin{lemma} \label{est}
For any $i>0$, if $r_{i-1} > r_i$, we have
$$
r_i \le \frac 13 d_i + \frac{d_{i-1}-d_i}6 + \frac 32.
$$
\end{lemma}

\begin{proof}
The lemma holds if
$$
r_i \le \frac 13 d_i + \frac 43.
$$
If not, then by Theorem \ref{castelnuovo2},
$$
r_i = \frac 13 d_i + \frac 53.
$$
In this case, $d_{i-1}-d_i \ge 1$. Otherwise, we would have
$$
r_{i-1} \ge r_i + 1= \frac 13 d_{i-1} + \frac 83.
$$
It contradicts with Theorem \ref{castelnuovo2}.
\end{proof}

\textit{From now on}, we assume that $X$ is an irregular minimal Gorenstein 3-fold of general type, $f: X \to Y$ is the fibration over a smooth curve $Y$ induced by the Albanese map, and $F$ is a smooth general fiber.
\begin{prop} \label{rough3}
If $p_g(F) > 0$ and $(p_g(F), K^2_F) \ne (2, 2)$, then
$$
\chi(\omega_X) \le  \left(\frac 16 + \frac{3}{2K^2_F} \right) K^3_{X}.
$$
\end{prop}

\begin{proof}
The proof is similar to Proposition \ref{rough2}. We sketch it here.
From Theorem \ref{relativenoether3fold2}, we get
$$
h^0(\omega_{X/Y}) \le \left(\frac 16 + \frac{3}{2K^2_F} \right) K^3_{X} - (K^2_F+9) (g(Y)-1) + \frac {K^2_F + 7}{3}.
$$
We still have
$$
h^0(\omega_{X/Y}) \ge \chi(\mathcal \omega_X) - (\chi(\mathcal O_F)+p_g(F)) (g(Y)-1).
$$
To prove the conclusion, by applying Remark \ref{observ}, it suffices to prove that
$$
K^2_F + 9 \ge \chi(\mathcal O_F)+p_g(F),
$$
which follows from the Noether inequality.
\end{proof}

\begin{proof}[Proof of Theorem \ref{alb12}]
If $p_g(F)=0$, then the theorem is true by Proposition \ref{pg0}. If $p_g(F) > 0$, by Proposition
\ref{rough3}, the theorem holds if $K^2_F \ge 9$.
\end{proof}

\begin{remark}
By the Castelnuovo inequality
$$
K^2_F \ge 3p_g(F) - 7,
$$
we know that $K^2_F \ge 9$ provided that $p_g(F) \ge 6$. In particular, it means that under the same assumption as in Theorem \ref{alb12}, we have
$$
K^3_X \ge 6 \chi(\omega_X)
$$
provided that $p_g(F) \ge 6$.
\end{remark}

The same as Theorem \ref{alb1}, one might guess that $K^2_F \ge 3$ will ``almost" imply that $K^3_X \ge 3 \chi(\omega_X)$.
But it could be probably optimistic. Anyway, it is true when $p_g(F)=0$ by Proposition \ref{pg0}. It is also true when
$p_g(F)=1$.

\begin{prop} \label{2pg1}
If $p_g(F)=1$ and $K^2_F > 2$, then
$$
K^3_X \ge 3\chi(\omega_X).
$$
\end{prop}
\begin{proof}
The proof is very similar to Proposition \ref{pg1}. We omit it here.
\end{proof}


\begin{thebibliography}{[AB]}

\bibitem{Ba1}
M. A. Barja, \textit{Lower bounds of the slope of fibred threefolds}, Internat. J. Math. {\bf 11} (2000), no. 4, 461--491.

\bibitem{Ba}
M. A. Barja, \textit{Generalized Clifford-Severi inequality and the volume of irregular varieties},
\href{http://arxiv.org/abs/1303.3045v2}{arXiv: 1303.3045v2}




\bibitem{BHPV}{W. Barth, K. Hulek, C. Peters, Van de Ven,
{\it Compact complex surfaces}, Second edition.
Ergebnisse der Mathematik und ihrer Grenzgebiete. 3.
Folge. A Series of Modern Surveys in Mathematics. 4.
Springer-Verlag, Berlin, 2004.}


\bibitem{Be}
A. Beauville,
\textit{L'application canonique pour les surfaces de type g\'en\'eral}, Invent. Math. \textbf{55}
(1979), 121--140.

\bibitem{BL}
Ch. Birkenhake, H. Lange,  \textit{Complex abelian varieties}, Grundlehren der Mathematischen Wissenschaften, vol. {\bf 302}.
Berlin, Heidelberg: Springer 1992.

\bibitem{Bo}
E. Bombieri, \textit{Canonical models of surfaces of general type}, Publ. Math. IHES \textbf{42} (1973), 171--219.


\bibitem{CaMChZ}
F. Catanese, M. Chen, D. -Q. Zhang, \textit{The Noether inequality for smooth minimal 3-folds}, Math. Res. Lett. {\bf 13} (2006), 653--666.

\bibitem{CC1}
J. A. Chen, M. Chen,
\textit{The Noether inequality for Gorenstein minimal 3-folds},
\href{http://arxiv.org/abs/1310.7709}{arXiv: 1310.7709}

\bibitem{CC}
J. A. Chen, M. Chen,
\textit{Explicit birational geometry of threefolds of general type, I}, Ann. Sci. Ec Norm. Sup. {\bf 43} (2010), 365--394.

\bibitem{CCZ}
J. A. Chen, M. Chen, D. -Q. Zhang, \textit{On the 5-canonical system of 3-folds of general type}, J. Reine Angew. Math. {\bf 603} (2007), 161--181.

\bibitem{CH}
J. A. Chen, C. Hacon,
\textit{On the geography of threefolds of general type},
J. Algebra {\bf 321} (2009), no. 9, 2500--2507.



\bibitem{Ch1}
M. Chen, \textit{Inequalities of Noether type for 3-folds of general type},
J. Math. Soc. Japan \textbf{56} (2004), no. 4, 1131--1155.

\bibitem{Ch}
M. Chen, \textit{Minimal threefolds of small slope and the Noether inequality for canonically polarized threefolds},
Math. Res. Lett. \textbf{11} (2004), no. 5--6, 833--852

\bibitem{MChH}
M, Chen, C. Hacon,
\textit{On the geography of Gorenstein minimal 3-folds of general type}, Asian J. Math. \textbf{10} (2006), no. 4, 757--763.

\bibitem{EL}
L. Ein, R. Lazarsfeld,
\textit{Singularities of theta divisors and the birational geometry of irregular varieties},
 J. Amer. Math. Soc. {\bf 10} (1997), no. 1, 243--258.

\bibitem{Fr}
P. Francia,
\textit{On the base points of the bicanonical system},
Symposia Math. \textbf{32} (1991), 141--150.

\bibitem{GL}
M. Green, R. Lazarsfeld,
\textit{Deformation theory, generic vanishing theorems, and some conjectures of Enriques, Catanese and Beauville},
Invent. Math. {\bf 90} (1987), no. 2, 389--407.



\bibitem{Ha}
C. Hacon,
\textit{A derived category approach to generic vanishing}, J. Reine Angew. Math. \textbf{575} (2004), 173--187.

\bibitem{HK}
C. Hacon, S. Kov\'acs,
\textit{Generic vanishing fails for singular varieties and in characteristic $p>0$}, \href{http://arxiv.org/abs/1212.5105}{arXiv: 1212.5105}.



\bibitem{Ho}
{E. Horikawa, {\it Algebraic surfaces of general type with small $c^2_1$. I}}, Ann. of Math. (2) {\bf 104} (1976), no. 2, 357--387.

\bibitem{Ho2}
E. Horikawa, \textit{Algebraic surfaces of general type with small
$c^{2}_{1}$. V}, J. Fac. Sci. Univ. Tokyo Sect. IA Math. \textbf{28}
(1981), 745--755.

\bibitem{YHu}
Y. Hu,
\textit{Inequality of Noether type for smooth minimal 3-folds of general type}, \href{http://arxiv.org/abs/1309.4618}{arXiv: 1309.4618}.

\bibitem{Hu}
B. Hunt,
\textit{Complex manifold geography in dimension $2$ and $3$},
J. Differential Geom. \textbf{30} (1989), no. 1, 51--153.


\bibitem{Ko1}
J. Koll\'ar,
\textit{Higher direct images of dualizing sheaves. I}, Ann. of Math. (2) {\bf 123} (1986), no. 1, 11--42.

\bibitem{Ko2}
J. Koll\'ar,
\textit{Higher direct images of dualizing sheaves. II}, Ann. of Math. (2) {\bf 124} (1986), no. 1, 171--202.

\bibitem{Li1}
C. Liedtke, \textit{Algebraic surfaces in positive characteristic}, to be published.

\bibitem{Li2}
C. Liedtke, \textit{Algebraic surfaces of general type with small $c^2_1$ in positive characteristic}, Nagoya Math. J. {\bf 191} (2008), 111--134.




\bibitem{MP}
M. Mendes Lopes, R. Pardini, \textit{The geography of irregular surfaces}, Current developments in algebraic geometry,  Math. Sci. Res. Inst. Publ. {\bf 59},
Cambridge Univ. Press (2012), 349--378.

\bibitem{MP2}
M. Mendes Lopes, R. Pardini, 
\textit{A uniform bound on the canonical degree of Albanese defective curves on surfaces}, 
Bull. Lond. Math. Soc. \textbf{44} (2012), no. 6, 1182--1188.

\bibitem{Lu1}
Steven S. Y. Lu, 
\textit{On surfaces of general type with maximal Albanese dimension}, 
J. Reine Angew. Math. \textbf{641} (2010), 163--175.

\bibitem{Lu2}
Steven S. Y. Lu, 
\textit{Holomorphic curves on irregular varieties of general type starting from surfaces}, Affine algebraic geometry, 205--220, CRM Proc. Lecture Notes, \textbf{54}, Amer. Math. Soc., Providence, RI, 2011.

\bibitem{Mi}
Y. Miyaoka, \textit{The Chern classes and Kodaira dimension of a minimal variety}, in Algebraic
Geometry, Sendai, 1985, Adv. Stud. Pure Math. {\bf 10} North-Holland, Amsterdam, (1987)
449--476.

\bibitem{Oh}
K. Ohno, \textit{Some inequalities for minimal fibrations of surfaces of general type over curves},
J. Math. Soc. Japan {\bf 44} (1992), 643--666.

\bibitem{Pa}
R. Pardini, \textit{The Severi inequality $K^2 \ge 4\chi$ for surfaces of maximal Albanese dimension}, Invent. Math. {\bf 159} (2005), no. 3, 669--672.



\bibitem{Se}
F. Severi, \textit{La serie canonica e la teoria delle serie principali de gruppi di punti sopra una superifcie algebrica},
Comment. Math. Helv. {\bf 4} (1932), 268--326.

\bibitem{Sh}
D. Shin, \textit{Noether inequality for a nef and big divisor on a surface}, Commun. Korean Math. Soc. \textbf{23} (2008),
no. 1, 11--18.


\bibitem{Si}
C. Simpson, \textit{Subspaces of moduli spaces of rank one local systems}, Ann. Sci. \'Ecole
Norm. Sup. (4) {\bf 26} (1993), 361--401.

\bibitem{Su}
H. Sun,
\textit{On the Clifford theorem for surfaces},
Tohoku Math. J. (2) \textbf{64} (2012), no. 2, 269--285.

\bibitem{Xi}
G. Xiao, \textit{Fibered algebraic surfaces with low slope},
Math. Ann. {\bf 276} (1987), no. 3, 449--466.


\bibitem{YZ1}
X. Yuan, T. Zhang,
{\it Effective bound of linear series on arithmetic surfaces}, Duke Math. J. {\bf 162} (2013), no. 10, 1723--1770.

\bibitem{YZ2}
X. Yuan, T. Zhang,
{\it Relative Noether inequality on fibered surfaces}, \href{http://arxiv.org/abs/1304.6122}{arXiv: 1304.6122}, submitted.

\bibitem{Zh}
T. Zhang,
{\it Severi inequality for varieties of maximal Albanese dimension}, \href{http://arxiv.org/abs/1303.4043}{arXiv: 1303.4043}, submitted.
\end{thebibliography}
\end{document}